\documentclass[11pt]{amsart}
\usepackage[colorlinks=true, pdfstartview=FitV, linkcolor=blue, citecolor=blue, urlcolor=blue]{hyperref}
\usepackage{amssymb,amsmath, amscd, array}
\usepackage{times, verbatim}
\usepackage{graphicx}
\usepackage[english]{babel}
 \usepackage[usenames, dvipsnames]{color}
\usepackage{amsmath,amssymb,amsfonts}
\usepackage{enumerate}
\usepackage{anysize}
\marginsize{3cm}{3cm}{3cm}{3cm}
\input xy
\xyoption{all}
\usepackage{pb-diagram}
\usepackage[all]{xy}
\input xy
\xyoption{all}

\DeclareFontFamily{OT1}{rsfs}{}
\DeclareFontShape{OT1}{rsfs}{n}{it}{<-> rsfs10}{}
\DeclareMathAlphabet{\mathscr}{OT1}{rsfs}{n}{it}

\begin{document}
\theoremstyle{plain}

\newtheorem{theorem}{Theorem}[section]
\newtheorem{thm}[equation]{Theorem}
\newtheorem{prop}[equation]{Proposition}
\newtheorem{corollary}[equation]{Corollary}
\newtheorem{conj}[equation]{Conjecture}
\newtheorem{lemma}[equation]{Lemma}
\newtheorem{definition}[equation]{Definition}
\newtheorem{question}[equation]{Question}

\theoremstyle{definition}
\newtheorem{conjecture}[theorem]{Conjecture}

\newtheorem{example}[equation]{Example}
\numberwithin{equation}{section}

\newtheorem{remark}[equation]{Remark}

\newcommand{\Hecke}{\mathcal{H}}
\newcommand{\Liea}{\mathfrak{a}}
\newcommand{\Cmg}{C_{\mathrm{mg}}}
\newcommand{\Cinftyumg}{C^{\infty}_{\mathrm{umg}}}
\newcommand{\Cfd}{C_{\mathrm{fd}}}
\newcommand{\Cinftyfd}{C^{\infty}_{\mathrm{ufd}}}
\newcommand{\sspace}{\Gamma \backslash G}
\newcommand{\PP}{\mathcal{P}}
\newcommand{\bfP}{\mathbf{P}}
\newcommand{\bfQ}{\mathbf{Q}}
\newcommand{\Siegel}{\mathfrak{S}}
\newcommand{\g}{\mathfrak{g}}
\newcommand{\A}{\mathbb{A}}
\newcommand{\la}{\mathbb{\langle}}
\newcommand{\ra}{\mathbb{\rangle}}
\def\G{{\rm G}}
\def\B{{\rm B}}
\def\T{{\rm T}}
\def\SL{{\rm SL}}
\def\PSL{{\rm PSL}}
\def\GSp{{\rm GSp}}

\newcommand{\wT}{\widehat{\T}}
\newcommand{\wG}{\widehat{\G}}
\newcommand{\wB}{\widehat{\B}}
\newcommand{\wrho}{\widehat{\rho}}
\newcommand{\Q}{\mathbb{Q}}
\newcommand{\Gm}{\mathbb{G}_m}
\newcommand{\Nm}{\mathbb{N}m}
\newcommand{\ii}{\mathfrak{i}}
\newcommand{\II}{\mathfrak{I}}

\newcommand{\kk}{\mathfrak{k}}
\newcommand{\nn}{\mathfrak{n}}
\newcommand{\tF}{\tilde{F}}
\newcommand{\p}{\mathfrak{p}}
\newcommand{\m}{\mathfrak{m}}
\newcommand{\bb}{\mathfrak{b}}
\newcommand{\Ad}{{\rm Ad}\,}
\newcommand{\ttt}{\mathfrak{t}}
\newcommand{\frakt}{\mathfrak{t}}
\newcommand{\U}{\mathcal{U}}
\newcommand{\Z}{\mathbb{Z}}
\newcommand{\bfG}{\mathbf{G}}
\newcommand{\bfT}{\mathbf{T}}
\newcommand{\R}{\mathbb{R}}
\newcommand{\ST}{\mathbb{S}}
\newcommand{\h}{\mathfrak{h}}
\newcommand{\bC}{\mathbb{C}}
\newcommand{\C}{\mathbb{C}}
\newcommand{\F}{\mathbb{F}}
\newcommand{\N}{\mathbb{N}}
\newcommand{\qH}{\mathbb {H}}
\newcommand{\temp}{{\rm temp}}
\newcommand{\Hom}{{\rm Hom}}
\newcommand{\Aut}{{\rm Aut}}
\newcommand{\rk}{{\rm rk}}
\newcommand{\Ext}{{\rm Ext}}
\newcommand{\End}{{\rm End}\,}
\newcommand{\Ind}{{\rm Ind}}
\newcommand{\ind}{{\rm ind}}
\newcommand{\Irr}{{\rm Irr}}
\def\circG{{\,^\circ G}}
\def\M{{\rm M}}
\def\diag{{\rm diag}}
\def\Ad{{\rm Ad}}
\def\As{{\rm As}}
\def\wG{{\widehat \G}}

\def\PGSp{{\rm PGSp}}
\def\Sp{{\rm Sp}}
\def\St{{\rm St}}
\def\GU{{\rm GU}}
\def\SU{{\rm SU}}
\def\U{{\rm U}}
\def\GO{{\rm GO}}
\def\GL{{\rm GL}}
\def\PGL{{\rm PGL}}
\def\GSO{{\rm GSO}}
\def\GSpin{{\rm GSpin}}
\def\GSp{{\rm GSp}}

\def\Gal{{\rm Gal}}
\def\SO{{\rm SO}}
\def\O{{\rm  O}}
\def\Sym{{\rm Sym}}
\def\sym{{\rm sym}}
\def\St{{\rm St}}
\def\Sp{{\rm Sp}}
\def\tr{{\rm tr\,}}
\def\ad{{\rm ad\, }}
\def\Ad{{\rm Ad\, }}
\def\rank{{\rm rank\,}}

\subjclass{Primary 20G05; Secondary 17B10, 22E46}

\title{Uniqueness of Branching through regular unipotent elements}

\author{Santosh Nadimpalli, Santosha Pattanayak}

\begin{abstract} Let \(\mathrm G\) be a complex simple algebraic group and let \(\mathrm G_0\subset \mathrm G\) be a closed connected subgroup containing a regular unipotent element of \(\mathrm G\), with semisimple rank at least \(2\). Using Dynkin's classification, we prove that the restriction of an irreducible finite-dimensional representation of \(\mathrm G\) to \(\mathrm G_0\) determines the representation up to an outer automorphism of \(\mathrm G\) preserving \(\mathrm G_0\).

We extend this method to the diagonal embedding
$\mathrm G_0\hookrightarrow \mathrm G\times \mathrm G$ for the specific pairs
$(\mathrm{SO}_{2k}(\mathbb C) \times\mathrm{SO}_{2k}(\mathbb C),\,\mathrm{SO}_{2k-1}(\mathbb C))$,
$(E_6\times E_6,\,F_4)$ and $(Spin_8(\mathbb C) \times Spin_8(\mathbb C), G_2)$ and show that uniqueness continues to hold. Finally, we give examples showing that, in the diagonal setting, restriction to the principal
\(\mathrm{SL}_2(\mathbb C)\) alone is not sufficient to establish uniqueness. \end{abstract}
\maketitle

  \section{Introduction}

  Let $\mathrm G$ be a connected simple complex algebraic group, and let 
${\mathrm G}_0 \subset \mathrm G$ be a closed connected subgroup containing a regular 
unipotent element of $\mathrm G$. In this paper, we study the problem of recovering an 
irreducible representation of $\mathrm G$ from its restriction to such a subgroup 
${\mathrm G}_0$.

The study of how representations of a group behave when restricted to its subgroups known as the branching problem, is a central theme in representation theory. 
Given an irreducible representation $V(\lambda)$ of ${\mathrm G}$ with highest weight 
$\lambda$, one seeks to understand the decomposition of its restriction 
$res_{\mathrm G_0}V(\lambda)$ into irreducible ${\mathrm G}_0$-modules. In general, the decomposition is highly nontrivial and depends in a subtle way on the chosen embedding ${\mathrm G}_0 \hookrightarrow {\mathrm G}$. The multiplicities occurring in this decomposition capture subtle structural and geometric features of both groups.

In this work, we approach the branching problem from a different perspective: instead of 
analyzing the full decomposition of the restricted representation, we ask how much of 
the original representation $V(\lambda)$ can be recovered solely from its restriction to 
${\mathrm G}_0$. Remarkably, when the subgroup ${\mathrm G}_0$ contains a regular unipotent element and has semisimple rank greater than one, the restriction of an irreducible representation determines the original representation uniquely up to outer automorphisms of $\mathrm G$.

Our main result gives an affirmative answer for all pairs in Dynkin's classification of simple subgroups containing a regular unipotent element, under the assumption that \(\operatorname{rank}\mathrm G_0\ge 2\). Thus the restriction of an irreducible representation of \(\mathrm G\) to \(\mathrm G_0\) determines the original representation up to an outer automorphism of \(\mathrm G\) preserving \(\mathrm G_0\). For example, two irreducible representations of $\SL_{2n}(\C)$ (resp., $\SL_{2n+1}(\C)$) which have the same restriction to  $\Sp_{2n}(\C)$ (resp., $\SO_{2n+1}(\C)$), are either isomorphic or are dual to each other. 
This was proved in our previous work
\cite{unique_branch} for 
some symmetric pairs $(\G, \G_0)$ using results of C. S. Rajan on the unique factorisation of characters
in the paper \cite{rajan_1}. 
The present proof is uniform and only uses the product 
formula established in \cite{character}. The argument covers all pairs $(\G,\G_0)$ in the Dynkin classification (see the list in Section~\ref{unique-single}) with $\operatorname{rank}\G_0\ge 2$,
including several that were not treated in \cite{unique_branch}, most
notably $(\mathrm{SL}_7(\mathbb C),\mathrm{G}_2)$ , $(\mathrm{SO}_7(\mathbb C),\mathrm{G}_2)$ and $(\mathrm{Spin}_7,\mathrm{G}_2)$.

The hypothesis on the subgroup ${\mathrm G}_0$ is crucial: the presence of a regular unipotent element ensures that ${\mathrm G}_0$ intersects every regular conjugacy class of $\mathrm G$ in a way that preserves the essential data of the representation. This setting includes, in particular, subgroups containing a 
\emph{principal} $\mathrm{SL}_2(\mathbb{C})$, which play a central role in the geometric 
and representation-theoretic structure of $\mathrm G$.

The proof relies on a detailed analysis of the restriction of characters to the image of 
the principal homomorphism
\[
\psi : \mathrm{SL}_2(\mathbb{C}) \longrightarrow \mathrm G.
\]
By evaluating the Weyl character formula on this one-parameter subgroup, we obtain the product formula for the character values in \cite{character} 
which captures the essential information of the representation in a single complex 
variable $z$. This specialization turns out to be sufficient to reconstruct the highest 
weight $\lambda$, and hence the representation $V(\lambda)$, up to outer automorphism.

One consequence of our result is that, for the Dynkin pairs considered here, branching is faithful on irreducible representations up to the natural outer symmetries of the ambient group. Thus, although the decomposition of \(\operatorname{res}_{\mathrm G_0}V(\lambda)\) into irreducible \(\mathrm G_0\)-modules may be complicated, the resulting \(\mathrm G_0\)-module still remembers the original \(\mathrm G\)-module. 

Extending the method to the diagonal embedding
$\mathrm G_0\hookrightarrow \mathrm G\times \mathrm G$, we establish uniqueness for the specific
pairs
$(\mathrm{SO}_{2k}(\mathbb C) \times\mathrm{SO}_{2k}(\mathbb C),\,\mathrm{SO}_{2k-1}(\mathbb C))$,$(E_6\times E_6,\,F_4)$ and $(Spin_8(\mathbb C) \times Spin_8(\mathbb C), G_2)$ (Section~\ref{two-copies}). We also give explicit examples showing that, in the diagonal setting, the principal \(\mathrm{SL}_2(\mathbb C)\)-data alone cannot replace the full restricted character to determine the original representations uniquely.

Questions in which Dynkin diagram automorphisms explain coincidences in
representation-theoretic constructions also arise in the work of Guilhot
and Lecouvey on induced modules \cite{guilhot_lecouvey}. The principal $\mathrm{SL}_2(\mathbb C)$ is another classical tool, originating in Kostant's
principal three-dimensional subgroup \cite{kostant_principal}; see also
Gross's work on minuscule representations and the principal $\mathrm{SL}_2(\mathbb C)$
\cite{gross_miniscule}. Related uses of special elements in character theory also appear in Prasad's work
\cite{dipendra_half_sum}. In the present paper, we use
principal $\mathrm{SL}_2(\mathbb C)$-specializations of characters to prove
uniqueness results for branching to subgroups containing regular
unipotent elements.

The paper is organized as follows. Section~\ref{prelim} recalls the product formula and the necessary background on root systems and principal \(\mathrm{SL}_2(\mathbb C)\)-subgroups. In Section~\ref{unique-single} we prove the main uniqueness theorem for single restrictions. Finally, Section~\ref{two-copies} treats diagonal embeddings, proves the diagonal uniqueness results listed above, and records examples illustrating the limitations of principal \(\mathrm{SL}_2(\mathbb C)\)-data in the diagonal problem.

\section{Preliminaries}\label{prelim} Let ${\rm G}$ be a connected complex reductive algebraic group with Lie algebra 
$\mathfrak{g}$. Fix a maximal torus $T \subset \mathrm G$ and a Borel subgroup 
$B$ containing it, i.e.\ $T \subset B$.

Let $\Phi$ denote the set of roots of $\mathrm G$ with respect to $T$, and let 
$\Phi^+ \subset \Phi$ be the subset of positive roots determined by the choice of 
$B$. Denote by $\Delta \subset \Phi^+$ the corresponding set of simple roots. Let $\mathfrak h=\operatorname{Lie}(T)$ and let $W$ be the Weyl group of $\mathrm G$.

Let $X^*(T)$ and $X_*(T)$ be, respectively, the character and 
cocharacter lattices of $T$, that is,
\[
X^*(T) = \Hom(T, \mathbb{C}^\times), 
\qquad 
X_*(T) = \Hom(\mathbb{C}^\times, T).
\]
These lattices are naturally in duality via the pairing 
$\langle \lambda, \mu \rangle = \deg(\lambda \circ \mu)$. We define \[ \rho = \frac{1}{2}\sum_{\alpha\in\Phi^+}\alpha \;\in\; X^*(T)\otimes_{\mathbb Z}\mathbb Q. \]
The set of coroots $\alpha^\vee$ of $\mathrm G$ lies in $X_*(T)$, and we define
the half-sum of positive coroots by
\[
\rho^\vee = \frac{1}{2} \sum_{\alpha \in \Phi^+} \alpha^\vee 
\in X_*(T) \otimes_\mathbb{Z} \mathbb{Q}.
\]
If \(\mathrm G_0\subset \mathrm G\) is a reductive subgroup and \(T\) is chosen so that $T_0:=T\cap \mathrm G_0$ is a maximal torus of \(\mathrm G_0\), then restriction of characters gives a map
\[
p \colon X^*(T) \longrightarrow X^*(T_0)
\]
which on the level of Lie algebras induces $p \colon \mathfrak{h}^* \to \mathfrak{h}_0^*$. It induces a homomorphism of group algebras
\[
p:\mathbb Z[P]\longrightarrow \mathbb Z[P_0] \,\, \,\,
\text{defined by} \,\, \,\, 
p(e^\lambda)=e^{p(\lambda)},
\] where $P$ and $P_0$ are the weight lattices of $\mathrm G$ and $\mathrm G_0$, respectively.
Under this map, the roots of $\mathrm G$ restrict to weights of $T_0$ that lie in
the root lattice of $\mathrm G_0$.  In the specific pairs we consider, each simple
root of $\mathrm G$ restricts to a simple root of $\mathrm G_0$.

For a dominant weight $\lambda \in X^*(T)$, we denote by $V(\lambda)$ the irreducible representation of $\mathrm G$ with highest weight $\lambda$ and by $\Theta_{\lambda}$ its character. Recall the Weyl character formula: 
\[
\Theta_{\lambda}
=
\frac{\displaystyle\sum_{w\in W}(-1)^{\ell(w)}e^{w(\lambda+\rho)}}
{\displaystyle\sum_{w\in W}(-1)^{\ell(w)}e^{w(\rho)}}.
\] We define the normalized Weyl numerator to be 
\[
U_\lambda
=
e^{-(\lambda+\rho)}
\sum_{w\in W}
(-1)^{\ell(w)}
e^{w(\lambda+\rho)}.
\]
Since $\lambda+\rho$ is dominant, we have 
\[
(\lambda+\rho)-w(\lambda+\rho)
=
\sum_{\alpha\in\Delta}
c_{w,\alpha}(\lambda)\alpha,
\qquad
c_{w,\alpha}(\lambda)\in \mathbb Z_{\ge 0}.
\]
So we may write
\[
U_\lambda
=
\sum_{w\in W}
(-1)^{\ell(w)}
\prod_{\alpha\in\Delta}
X_\alpha^{c_{w,\alpha}(\lambda)},
\]
and therefore $U_\lambda\in 
R:=\mathbb Z[X_\alpha:\alpha\in\Delta] \subset \mathbb Z[P]$, where $X_\alpha:=e^{-\alpha}$.

There exists a homomorphism of algebraic groups
\begin{equation}\label{principal}
  \psi : \mathrm{SL}_2(\mathbb{C}) \longrightarrow \mathrm G,
\end{equation}
which sends a regular unipotent element of $\mathrm{SL}_2(\mathbb{C})$ to a regular 
unipotent element of $\mathrm G$. Such a homomorphism is unique up to conjugacy in 
$\mathrm G$ and is called the \emph{principal} $\mathrm{SL}_2(\mathbb{C})$ in $\mathrm G$. 

The following formula proved in \cite{character} describes the restriction of the character 
$\Theta_\lambda$ of $V(\lambda)$ to the 
\emph{principal} $\mathrm{SL}_2(\mathbb{C})$ inside $\mathrm G$. 
It may be viewed as a specialization of the Weyl character formula obtained by 
evaluating it on the one-parameter subgroup 
$z \mapsto \psi(\mathrm{diag}(z, z^{-1}))$.

\begin{theorem}\label{prod}
    
Let $\mathrm G$ be a connected complex reductive group, and let 
\[
\psi : \mathrm{SL}_2(\mathbb{C}) \longrightarrow \mathrm G
\]
be the \emph{principal homomorphism} as in~\eqref{principal}. For 
$z \in \mathbb{C}^\times$, define
\[
\Theta_\lambda(z)
\;=\;
\Theta_\lambda\big(\psi(\mathrm{diag}(z, z^{-1}))\big),
\]
that is, the value of the character on the image under $\psi$ of the diagonal 
element $\mathrm{diag}(z, z^{-1}) \in \mathrm{SL}_2(\mathbb{C})$. Then the following \emph{product formula} holds:
\begin{equation}\label{productformula}
\Theta_\lambda(z)
\;=\;
z^{-2\langle \lambda, \rho^\vee \rangle}\,
\frac{\displaystyle \prod_{\alpha \in \Phi^+}
\big(1 - z^{2\langle \lambda + \rho,\, \alpha^\vee \rangle}\big)}
{\displaystyle \prod_{\alpha \in \Phi^+}
\big(1 - z^{2\langle \rho,\, \alpha^\vee \rangle}\big)},
\end{equation}

where both the numerator and the denominator are polynomials in $z^2$;
the prefactor $z^{-2\langle \lambda, \rho^\vee \rangle}$ may contribute an odd power of $z$.

\end{theorem}

\section{Uniqueness of branching}\label{unique-single}
Let $\G$ be a complex simple algebraic group.  Let
$\G_0$ be a simple algebraic subgroup of $\G$ such that 
$\G_0$ contains a regular unipotent element of $\G$. The classification of such pairs $(\G, \G_0)$, goes back to Dynkin, see
\cite[Chapter IX, Exercise 20]{Bo}. 
\begin{enumerate}
\item ${\rm Sp}_{2n}(\mathbb{C})\subset 
{\rm SL}_{2n}(\mathbb{C})$,
\item ${\rm SO}_{2n+1}(\mathbb{C})
\subset {\rm SL}_{2n+1}(\mathbb{C})$,
\item ${\rm SO}_{2n+1}(\mathbb{C})
\subset  {\rm SO}_{2n+2}(\mathbb{C})$,
\item ${\rm \G}_2(\mathbb{C})\subset
{\rm Spin}_7(\mathbb{C})\subset
{\rm SO}_8(\mathbb{C})$,
\item ${\rm G}_2(\mathbb{C})\subset
{\rm SO}_7(\mathbb{C})\subset
{\rm SL}_7(\mathbb{C})$ and 
\item $F_4\subset E_6$. 
    \end{enumerate}

    The embedding of $\mathrm{Spin}_7(\mathbb{C})$ in $\mathrm{SO}_8(\mathbb{C})$ is given by the 8-dimensional spin representation of $\mathrm{Spin}_7(\mathbb{C})$. The embedding of $\mathrm{G}_2(\mathbb{C})$ into $\mathrm{Spin}_8(\mathbb{C})$ arises as the fixed-point subgroup of the triality automorphism of $\mathrm{Spin}_8(\mathbb{C})$. In terms of representations, the 8-dimensional spin representation of $\mathrm{Spin}_7(\mathbb{C})$ restricts to $\mathrm{G}_2(\mathbb{C})$ as the direct sum of the trivial representation and the 7-dimensional irreducible representation, yielding the embedding $\mathrm{G}_2(\mathbb{C}) \subset \mathrm{SO}_8(\mathbb{C})$. Furthermore, this 7-dimensional irreducible representation of $\mathrm{G}_2(\mathbb{C})$ preserves a quadratic form, yielding the embedding $\mathrm{G}_2(\mathbb{C}) \subset \mathrm{SO}_7(\mathbb{C})$, which sits naturally inside $\mathrm{SL}_7(\mathbb{C})$. The remaining embeddings are obtained as fixed points of an involution.

    In this section, for two dominant weights $\lambda$ and $\mu$ of $\G$, we classify all pairs 
$(\pi_\lambda, \pi_\mu)$ of irreducible finite–dimensional representations of $\G$ 
satisfying 
$$
\operatorname{res}_{G_0} \pi_\lambda \simeq \operatorname{res}_{G_0} \pi_\mu.
$$
We will show that $\pi_\lambda$ and $\pi_\mu$ coincide up to a possible outer automorphism of $\G$ that fixes $\G_0$. We fix a maximal torus and Borel subgroup 
$\T \subset \B$ of $\G$ such that $\T \cap \G_0$ is a maximal torus of $\G_0$ and 
$\B \cap \G_0$ is a Borel subgroup of $\G_0$. Let $\mathfrak g$ and $\mathfrak g_0$ denote the Lie algebras of $\G$ and $\G_0$ respectively. Let $\mathfrak h$ be the Lie algebra of $\T$, and set 
$\mathfrak h_0 = \mathfrak g_0 \cap \mathfrak h$, a Cartan subalgebra of $\mathfrak g_0$. 

Let $\Phi$ be the root system of $\mathfrak g$ with respect to $\mathfrak h$, and fix a basis 
$\Delta = \{\alpha_1, \alpha_2, \ldots, \alpha_n\}$ of simple roots. For each root 
$\alpha \in \Phi$, choose a root vector $X_\alpha$ so that 
$\{ X_\alpha : \alpha \in \Phi \}$ forms a Chevalley basis of $\mathfrak g$.

The main theorem of this section is as follows.

    \begin{theorem}\label{res_thm}
Let $(\G, \G_0)$ be a pair of simple complex algebraic groups with 
$\G_0 \subset \G$, and assume that $\G_0$ contains a regular unipotent 
element of $\G$. Assume moreover that \(\operatorname{rank}(\G_0)\ge 2\). 

Let $\lambda$ and $\mu$ be dominant weights in $X^\ast(\T)$ (with respect to the chosen Borel subgroup $\B$) such that
\[
\operatorname{res}_{\G_0} \pi_\lambda \simeq 
\operatorname{res}_{\G_0} \pi_\mu.
\]
Then there exists an automorphism
\[
\sigma \in \operatorname{Aut}(\G, \B, \T, \{X_\alpha\})
\]
such that
\[
\sigma(\G_0 \cap \T) = \G_0 \cap \T
\qquad\text{and}\qquad
\sigma(\lambda) = \mu.
\]
\end{theorem}

\begin{proof} We view $\pi_\lambda$ and $\pi_\mu$ as finite-dimensional representations
of the Lie algebra $\mathfrak g$. 
For a dominant weight $\nu$, set
\[
u_\nu:=p(U_\nu),
\]
where
\[
U_\nu
=
e^{-(\nu+\rho)}
\sum_{w\in W}(-1)^{\ell(w)}e^{w(\nu+\rho)}
\]
is the normalized Weyl numerator.

By \cite[Section~4.1, Lemma~4.4]{unique_branch}, the hypothesis
\[
\operatorname{res}_{\G_0}\pi_\lambda
\simeq
\operatorname{res}_{\G_0}\pi_\mu
\]
is equivalent to
\[
u_\lambda=u_\mu.
\tag{1}
\]

We first treat the cases in which $\G_0$ is obtained as the fixed-point
subgroup of a diagram automorphism of $\G$. These are the folding cases
\[
\mathrm{Sp}_{2n}\subset \mathrm{SL}_{2n},\qquad
\mathrm{SO}_{2n+1}\subset \mathrm{SL}_{2n+1},\qquad
\mathrm{SO}_{2n+1}\subset \mathrm{SO}_{2n+2},
\]
the triality case
\[
G_2\subset \mathrm{Spin}_8,
\]
and the exceptional folding
\[
F_4\subset E_6.
\]
In these cases we may write $
\mathfrak g_0=\mathfrak g^\theta$
for a nontrivial diagram automorphism $\theta$ of $\mathfrak g$. 

For these fixed-point cases, we prove the assertion by induction on the
rank of \(\mathfrak g\). The induction hypothesis is the statement of the
theorem for all fixed-point pairs of strictly smaller rank.

Choose a decomposition
\[
\Delta=\Delta_1\sqcup \Delta_2
\]
such that $p(\Delta_1)$ is connected and $p(\Delta_2)$ consists of a
single simple root. For $i\in\{1,2\}$, let $\Phi_i$ be the root subsystem
generated by $\Delta_i$, let $\mathfrak g^{(i)}$ be the corresponding
semisimple Lie subalgebra, and set
\[
\mathfrak g_0^{(i)}:=\mathfrak g^{(i)}\cap \mathfrak g_0.
\]
Then $(\mathfrak g^{(i)},\mathfrak g_0^{(i)})$ is again a
finite-order fixed-point pair of strictly smaller rank. In the cases
where $\mathfrak g^{(i)}$ is not simple, the induction hypothesis is
applied to the corresponding semisimple fixed-point pair; the rank-one
factors are treated as the base cases.

Let $\{\varpi'_\alpha:\alpha\in\Delta_i\}$
be the fundamental weights of the root system $\Phi_i$, with respect to
$\mathfrak g^{(i)}\cap\mathfrak h$ and
$\mathfrak g^{(i)}\cap\mathfrak b$. We define
\[
\lambda^{(i)}
=
\sum_{\alpha\in\Delta_i}
\langle \lambda+\rho,\alpha^\vee\rangle\varpi'_\alpha,
\qquad
\mu^{(i)}
=
\sum_{\alpha\in\Delta_i}
\langle \mu+\rho,\alpha^\vee\rangle\varpi'_\alpha,
\]
and let 
\[
\rho_i:=\sum_{\alpha\in\Delta_i}\varpi'_\alpha.
\]

Let $I=\Delta_i$. We define the polynomial projection
\[
\pi_I:\mathbb Z[X_\alpha:\alpha\in\Delta]
\longrightarrow
\mathbb Z[X_\alpha:\alpha\in I]
\]
by
\[
\pi_I(X_\alpha)
=
\begin{cases}
X_\alpha,&\alpha\in I,\\
0,&\alpha\notin I.
\end{cases}
\]
Similarly we define
\[
\pi_I^0:
\mathbb Z[X_{\beta'}:\beta'\in\Delta(\T_0,\B_0)]
\longrightarrow
\mathbb Z[X_{\beta'}:\beta'\in p(I)]
\]
by
\[
\pi_I^0(X_{\beta'})
=
\begin{cases}
X_{\beta'},&\beta'\in p(I),\\
0,&\beta'\notin p(I).
\end{cases}
\]

Let $
p_i:\mathbb Z[X_\alpha:\alpha\in I]
\longrightarrow
\mathbb Z[X_{\beta'}:\beta'\in p(I)]$
be the restriction map for the smaller pair
$(\mathfrak g^{(i)},\mathfrak g_0^{(i)})$. For the subsets \(I=\Delta_i\) chosen below, we have
\[
p(I)\cap p(\Delta\setminus I)=\varnothing.
\]
Hence the maps satisfy
\[
\pi_I^0\circ p=p_i\circ \pi_I.
\tag{2}
\]
Now applying $\pi_I^0$ to the equality \((1)\), we get 
\[
\pi_I^0(p(U_\lambda))
=
\pi_I^0(p(U_\mu)).
\]
Using \((2)\), this becomes
\[
p_i(\pi_I(U_\lambda))
=
p_i(\pi_I(U_\mu)).
\tag{3}
\]

By \cite[Section 4.1, Lemma~4.5]{unique_branch},
\[
\pi_I(U_\lambda)
=
U^{(i)}_{\lambda^{(i)}-\rho_i},
\qquad
\pi_I(U_\mu)
=
U^{(i)}_{\mu^{(i)}-\rho_i}.
\]
Thus \((3)\) gives
\[
p_i\left(U^{(i)}_{\lambda^{(i)}-\rho_i}\right)
=
p_i\left(U^{(i)}_{\mu^{(i)}-\rho_i}\right).
\tag{4}
\]

Applying \cite[Section~4.1, Lemma~4.4]{unique_branch} to the smaller pair
$(\mathfrak g^{(i)},\mathfrak g_0^{(i)})$, we obtain
\[
\operatorname{res}_{\mathfrak g_0^{(i)}}
\pi_{\lambda^{(i)}-\rho_i}
\simeq
\operatorname{res}_{\mathfrak g_0^{(i)}}
\pi_{\mu^{(i)}-\rho_i}.
\tag{5}
\]

Since
\[
\operatorname{rank}\mathfrak g^{(i)}
<
\operatorname{rank}\mathfrak g,
\]
the induction hypothesis applies to the pair
$(\mathfrak g^{(i)},\mathfrak g_0^{(i)})$. Hence there exists an
automorphism
\[
\sigma_i\in
\operatorname{Aut}(\G^{(i)},\B_i,\T_i,\{X_\alpha:\alpha\in\Phi_i\})
\]
such that
\[
\sigma_i(\lambda^{(i)}-\rho_i)
=
\mu^{(i)}-\rho_i.
\]
Since every diagram automorphism preserves
\[
\rho_i=\sum_{\alpha\in\Delta_i}\varpi'_\alpha,
\]
we get
\[
\sigma_i(\lambda^{(i)})=\mu^{(i)}.
\]
Since \(\sigma_i\) belongs to the corresponding automorphism group
preserving the fixed-point data, the representations
\(\pi_{\lambda^{(i)}}\) and \(\pi_{\sigma_i(\lambda^{(i)})}\) have
isomorphic restrictions to \(\mathfrak g_0^{(i)}\). Hence
\begin{equation}\label{tail}
\operatorname{res}_{\mathfrak g_0^{(i)}}
\pi_{\lambda^{(i)}}
\simeq
\operatorname{res}_{\mathfrak g_0^{(i)}}
\pi_{\mu^{(i)}},
\qquad i=1,2.
\end{equation}
These isomorphisms are used in the corresponding case-by-case
analysis for the fixed-point pairs.
In Dynkin's list, the remaining cases are
\[
G_2\subset \mathrm{Spin}_7,\qquad
G_2\subset SO_7,\qquad
G_2\subset SL_7.
\]
The embedding \(G_2\subset \mathrm{Spin}_7\) is the stabilizer-type
embedding arising from the \(8\)-dimensional spin representation of
\(\mathrm{Spin}_7\), while \(G_2\subset SO_7\) is given by the
\(7\)-dimensional irreducible representation of \(G_2\). The embedding
\(G_2\subset SL_7\) is obtained by composing
\[
G_2\subset SO_7\subset SL_7.
\]
These embeddings are not fixed-point subgroups of diagram automorphisms
of \(\mathrm{Spin}_7\), \(SO_7\), or \(SL_7\). Hence the preceding
induction argument and the isomorphism \eqref{tail} do not apply to
them directly. We therefore treat these cases separately by explicit
root-theoretic computations.

Let \(A=\{n_1,n_2,\dots,n_k\}\) and \(B=\{m_1,m_2,\dots,m_k\}\) be finite multisets of positive integers.  
Observe that
\[
\prod_{i=1}^k \frac{1-t^{n_i}}{1-t}
=
\prod_{j=1}^k \frac{1-t^{m_j}}{1-t}
\qquad\Longleftrightarrow\qquad
A=B.
\]

Since 
\(\operatorname{res}_{\G_0} \pi_\lambda \simeq 
\operatorname{res}_{\G_0} \pi_\mu\),
and because the principal homomorphism 
\(\mathrm{SL}_2(\mathbb{C}) \to \G\) factors through \(\G_0\),  
Theorem~\ref{prod} implies the equality of multisets
\begin{equation}\label{principal_numbers}
\{\langle \lambda+\rho, \alpha^\vee \rangle : \alpha \in \Phi^+\}
=
\{\langle \mu+\rho, \alpha^\vee \rangle : \alpha \in \Phi^+\}.
\end{equation}

We now consider the various possibilities for the pair $(\G,\G_0)$.
Throughout, we use the Bourbaki numbering of simple roots
\cite[Planche]{bourbaki_lie_4_6}, and we write
\[
n_i=\langle \lambda+\rho,\alpha_i^\vee\rangle,
\qquad
m_i=\langle \mu+\rho,\alpha_i^\vee\rangle .
\]

\subsection*{Case 1:} $\G=\mathrm{SL}_{k+1}(\mathbb{C})$ and 
\[
\G_0 =
\begin{cases}
\mathrm{Sp}_{2n}(\mathbb{C}) &\text{if } k=2n-1,\\[2mm]
\mathrm{SO}_{2n+1}(\mathbb{C}) &\text{if } k=2n.
\end{cases}
\]
Let $\Delta=\{\alpha_1,\dots,\alpha_k\}$ and set
\[
\Delta_1=\{\alpha_2,\alpha_3,\dots,\alpha_{k-1}\},
\qquad
\Delta_2=\{\alpha_1,\alpha_k\}.
\]

We prove that either $\lambda=\mu$ or  
$\lambda=\sigma(\mu)$ for the nontrivial diagram automorphism  
$\sigma\in \mathrm{Aut}(\G,\B,\T,\{X_\alpha\})$.
The argument proceeds by induction on $k$, the claim being trivial for 
$k=2$.  Assume that it holds for $k-1$.

After replacing \(\mu\) by its image under the nontrivial diagram
automorphism of \(A_k\), if necessary, we may assume that $\lambda^{(1)}=\mu^{(1)}$. That is
\[
n_i=m_i,\qquad 2\le i\le k-1.
\]
Moreover, the rank-one tail gives
\[
\{n_1,n_k\}=\{m_1,m_k\}.
\]
Using the equality 
\eqref{principal_numbers},
we get the equality of following multsets:
\begin{align*}
  \{n_1, n_1+n_2,\dots, n_1+n_2+\cdots+n_{k-1}+n_k,
n_k+n_{k-1}+\cdots+n_2,\dots, n_k+n_{k-1}, n_k\}= \\
\{m_1, m_1+n_2,\dots, m_1+n_2+\cdots+n_{k-1}+m_k,
m_k+n_{k-1}+\cdots+n_2,\dots, m_k+n_{k-1}, m_k\}. 
\end{align*}
Assume that $n_1\neq m_1$. 
Then, we get that   
\begin{align*}
  \{n_1+n_2, n_1+n_2+n_3,\dots, n_1+\cdots+n_{k-1},
n_k+\cdots+n_2,\dots, n_k+n_{k-1}\}= \\
\{n_k+n_2,n_k+n_2+n_3,\dots, n_k+n_2+\cdots+n_{k-1},
n_1+n_{k-1}+\cdots+n_2,\dots, n_1+n_{k-1}\}. 
\end{align*}
Now, we have
$${\rm min}\{n_1+n_2, n_k+n_{k-1}\}
={\rm min}\{n_k+n_2, n_1+n_{k-1}\}.$$
Since $n_1\neq n_k$,
we conclude that $n_2=n_{k-1}$. 
Then we get that 
$${\rm min}\{n_1+n_2+n_3, n_k+n_{k-1}+n_{k-2}\}
={\rm min}\{n_k+n_2+n_3, n_1+n_{k-1}+n_{k-2}\}.$$
which is the same as 
$${\rm min}\{n_1+n_2+n_3, n_k+n_2+n_{k-2}\}
={\rm min}\{n_k+n_2+n_3, n_1+n_2+n_{k-2}\}.$$
Which implies that $n_3=n_{k-2}$. Using induction on the integer
$i$, we conclude that $n_i=n_{k-i+1}$, for all $2\leq i\leq k-1$. 
Thus, we get that $\pi_\lambda$ is either isomorphic to $\pi_\mu$
as ${\rm SL}_{k+1}(\mathbb C)$ representations or $\pi_\lambda\simeq \pi_\mu^\vee$. 

\subsection*{Case 2:} $\G=\mathrm{SO}_{2n+2}(\mathbb{C})$ and $\G_0=\mathrm{SO}_{2n+1}(\mathbb{C})$ for $n\ge1$. The case \(n=2\), namely \((\mathrm{SO}_6(\mathbb C),\mathrm{SO}_5(\mathbb C))\), follows from Case~1, since \[
\mathfrak{so}_6\simeq\mathfrak{sl}_4,
\qquad
\mathfrak{so}_5\simeq\mathfrak{sp}_4,
\]

We again use induction on $n$, assuming the statement is known for
$n-1$. Let
\[
\Delta_1=\Delta\setminus\{\alpha_1\},
\qquad
\Delta_2=\{\alpha_1\}.
\]
From \eqref{tail}, the weights $\lambda^{(1)}$ and $\mu^{(1)}$
agree up to a possible diagram automorphism of type $D_{n}$, and
\[
n_1=m_1.
\]
Thus $\lambda=\sigma(\mu)$ for some
$\sigma\in \mathrm{Aut}(\G,\B,\T,\{X_\alpha\})$, completing the proof in
this case.

\subsection*{Case 3:} We now consider the case where $\G=\mathrm{Spin}_7(\mathbb{C})$ and 
$\G_0=G_2$. Let $\pi_\lambda$ and $\pi_\mu$ be two irreducible representations of $\G$ such that 
\[
\operatorname{res}_{\G_0}\, \pi_\lambda \simeq 
\operatorname{res}_{\G_0}\, \pi_\mu .
\]
 The subgroup $\G_0$ contains a regular unipotent element of $\G$;
hence the principal $\mathrm{SL}_2(\mathbb C)$ of $\G$ factors through $\G_0$.
Therefore
\[
\operatorname{res}_{\G_0} V_\lambda \simeq \operatorname{res}_{\G_0} V_\mu
\;\Longrightarrow\;
\operatorname{res}_H V_\lambda \simeq \operatorname{res}_H V_\mu ,
\]
where $H$ is the image of the principal homomorphism $\psi:\mathrm{SL}_2(\mathbb C) \to\G$. 
From the equality \eqref{principal_numbers} we get that the multiset $$\{2n_1+2n_2+n_3, n_1+2n_2+n_3, 
2n_2+n_3, n_1+n_2+n_3, n_2+n_3, n_1+n_2, n_1, n_2, n_3\}$$
is equal to the multiset 
$$\{2m_1+2m_2+m_3, m_1+2m_2+m_3, 
2m_2+m_3, m_1+m_2+m_3, m_2+m_3, m_1+m_2, m_1, m_2, m_3\}.$$ The largest two elements of the principal \(\mathrm{SL}_2\)-multiset are
\[
2n_1+2n_2+n_3
\qquad\text{and}\qquad
n_1+2n_2+n_3.
\]
Comparing them with the corresponding \(m\)-expressions gives
\[
n_1=m_1,\qquad 2n_2+n_3=2m_2+m_3.
\]
Comparing the sums of all elements in the two multisets gives
\[
10n_2+6n_3=10m_2+6m_3.
\]
These two linear equations imply \(n_2=m_2\) and \(n_3=m_3\). Hence
\(\lambda=\mu\).

\subsection*{Case 4} We now consider the case where $\G = \mathrm{SO}_7(\mathbb{C})$ and 
$\G_0 = G_2$.
 This case reduces to Case~3 $(\mathrm{Spin}_7,\mathrm{G}_2)$ as follows.  
The simply connected cover $\pi:\mathrm{Spin}_7\to\mathrm{SO}_7$ is an isogeny of degree $2$.  
The inclusion $\mathrm{G}_2\subset\mathrm{SO}_7$ lifts uniquely to $\mathrm{G}_2\subset\mathrm{Spin}_7$ because $\mathrm{G}_2$ is simply connected.  
Every finite‑dimensional representation of $\mathrm{SO}_7$ lifts via $\pi$ to a representation of $\mathrm{Spin}_7$.  
If $\lambda,\mu$ are dominant weights of $\mathrm{SO}_7$ with $\operatorname{res}_{\mathrm{G}_2}\pi_\lambda\simeq\operatorname{res}_{\mathrm{G}_2}\pi_\mu$, then the corresponding lifted representations $\tilde\pi_\lambda,\tilde\pi_\mu$ of $\mathrm{Spin}_7$ have isomorphic restrictions to $\mathrm{G}_2$.  
Then by Case~3, the lifted highest weights are equal. Therefore the original
highest weights for \(SO_7\) are equal, and hence $
\pi_\lambda\simeq\pi_\mu.$

\subsection*{Case 5:} We now consider the case where $\G = \mathrm{SO}_8(\mathbb{C})$ and 
$\G_0 = \mathrm{Spin}_7(\mathbb{C})$, where the embedding of $\G_0$ into $\G$
is induced by the $8$-dimensional spin representation.  Since this 
representation does not factor through $\mathrm{SO}_7(\mathbb{C})$, we may, and
do, regard $\mathrm{Spin}_7(\mathbb{C})$ as a subgroup of 
$\mathrm{Spin}_8(\mathbb{C})$, the simply connected cover of $\mathrm{SO}_8$.

The advantage of this setting is that the group $\mathrm{Spin}_8(\mathbb{C})$ has triality automorphisms 
$\theta \in \mathrm{Out}(\mathrm{Spin}_8)$ $\cong S_3$, which act transitively on
the three $8$-dimensional irreducible representations: the vector
representation and the two half-spin representations.  
The subgroup $\mathrm{Spin}_7(\mathbb{C})$ embeds in 
$\mathrm{Spin}_8(\mathbb{C})$ in such a way that one of these three 
representations restricts to the $8$-dimensional irreducible representation 
of $\mathrm{Spin}_7(\mathbb{C})$.  
Applying a suitable element of $\mathrm{Out}(\mathrm{Spin}_8)$, we may therefore
assume that this representation is the vector representation of 
$\mathrm{Spin}_8(\mathbb{C})$.  
In particular, up to a triality automorphism, we may reduce to the standard
embedding 
\[
\mathrm{Spin}_7(\mathbb{C}) \subset \mathrm{Spin}_8(\mathbb{C}).
\]
Under this standard embedding, the restriction of highest weights and the corresponding
principal $\mathrm{SL}_2$--weights are identical to the situation considered
earlier for the pair
\[
\SO(2n+1) \subset \SO(2n+2)
\quad\text{with } n=3.
\]
Thus all arguments from the previous case apply verbatim, and we conclude that
\[
\mathrm{res}_{\G_0}\pi_\lambda \simeq 
\mathrm{res}_{\G_0}\pi_\mu
\quad\Longrightarrow\quad
\lambda = \sigma(\mu)
\]
for some $\sigma \in \mathrm{Aut}(\G,\B,\T,\{X_\alpha\})$.

\subsection*{Case 6:} In this case, we take $\G$ to be the simple algebraic group of type $E_6$, and 
$\G_0$ to be the subgroup of type $F_4$ fixed by the non-trivial diagram automorphism in 
${\rm Aut}(\G,\B,\T,\{X_\alpha\})$.  
We decompose the set of simple roots as
\[
\Delta_1=\Delta\setminus\{\alpha_2\}, \qquad   
\Delta_2=\{\alpha_2\}.
\]
Applying equation~\eqref{tail}, we obtain
\[
n_2 = m_2, \qquad 
\lambda^{(1)} = \mu^{(1)} \ \text{or}\ \lambda^{(1)} = \sigma(\mu^{(1)}),
\]
where $\sigma$ is the restriction to the subsystem
\(\Delta_1\) of the nontrivial diagram automorphism of $E_6$ and this automorphism
fixes \(\alpha_2\). Since \(n_2=m_2\), the equality on \(\Delta_1\) together
with the equality at \(\alpha_2\) implies
\[
\lambda=\sigma(\mu)
\]
for \(\sigma=1\) or for the nontrivial diagram automorphism of \(E_6\). 

\subsection*{Case 7:} We now consider the case where $\G = \mathrm{SO}_8(\mathbb{C})$ and $\G_0 = \mathrm{G}_2(\mathbb{C})$. 
Let $\pi_\lambda$ and $\pi_\mu$ be two irreducible representations of $\G$ such that 
\[
\operatorname{res}_{\G_0}\, \pi_\lambda \simeq 
\operatorname{res}_{\G_0}\, \pi_\mu .
\]
 The subgroup $\G_0$ contains a regular unipotent element of $\G$;
hence the principal $\mathrm{SL}_2(\mathbb C)$ of $\G$ factors through $\G_0$.
Therefore
\[
\operatorname{res}_{\G_0} V_\lambda \simeq \operatorname{res}_{\G_0} V_\mu
\;\Longrightarrow\;
\operatorname{res}_H V_\lambda \simeq \operatorname{res}_H V_\mu ,
\]
where $H$ is the image of the principal homomorphism $\psi:\mathrm{SL}_2(\mathbb C) \to\G$. 
From the equality \eqref{principal_numbers} we get that the multiset
   \begin{align*}
       \{n_1, n_2, n_3, n_4, n_1+n_2, n_3+n_2, n_4+n_2, n_1+n_2+n_3, n_3+n_2+n_4,\\
   n_4+n_2+n_1, n_1+n_2+n_3+n_4, 
   n_1+2n_2+n_3+n_4\}
   \end{align*}
   is equal to the multiset
   \begin{align*}
       \{m_1, m_2, m_3, m_4, m_1+m_2, m_3+m_2, m_4+m_2, m_1+m_2+m_3, m_3+m_2+m_4,\\
   m_4+m_2+m_1, m_1+m_2+m_3+m_4, 
   m_1+2m_2+m_3+m_4\}.
   \end{align*}

Comparing the largest and second largest elements in these multisets, we conclude that  
\[
n_2 = m_2 \,\, \text{and also} \,\, 
n_1 + n_3 + n_4 = m_1 + m_3 + m_4 .
\]
Since triality permutes the outer nodes, we may assume
\[
n_1\le n_3\le n_4,\qquad m_1\le m_3\le m_4.
\]
After deleting the common occurrence \(n_2=m_2\) from the two equal
multisets, the smallest remaining elements are \(n_1\) and \(m_1\);
hence \(n_1=m_1\). Deleting now the common occurrences \(n_1=m_1\) and
\(n_1+n_2=m_1+m_2\), the smallest remaining elements are \(n_3\) and
\(m_3\), so \(n_3=m_3\). Finally, from
\[
n_1+n_3+n_4=m_1+m_3+m_4
\]
we get \(n_4=m_4\). Thus
\[
\{n_1,n_3,n_4\}=\{m_1,m_3,m_4\}.
\]
Thus $\lambda$ and $\mu$ differ only by a permutation of the simple roots 
$\alpha_1$, $\alpha_3$ and $\alpha_4$, i.e.,
\[
\pi_\lambda \simeq \pi_\mu^\sigma,
\]
where $\sigma$ is a diagram automorphism of $\G$.

Finally, note that the subgroup $\G_2 \subset \G$ is fixed by the full diagram automorphism group ${\rm Aut}(\G,\B,\T,\{X_\alpha\})$. Therefore, any two irreducible finite-dimensional representations of $\G$ with isomorphic restriction to $\G_2$ must be isomorphic up to an outer automorphism of $\G$.

 \subsection*{Case 8:} Finally we come to the case where $\G={\rm SL}_7(\mathbb C)$
   and $\G_0=\G_2 \subseteq {\rm SL}_7(\mathbb C)$. Let \(\Delta=\{\alpha_1,\dots,\alpha_6\}\) be the simple roots of \(A_6\), and write
\[
n_i=\langle \lambda+\rho,\alpha_i^\vee\rangle,\qquad
m_i=\langle \mu+\rho,\alpha_i^\vee\rangle \qquad (1\le i\le 6).
\]
Let \(\{\beta_1,\beta_2\}\) be the simple roots of \(G_2\), with \(\beta_1\) short and \(\beta_2\) long. For the standard \(7\)-dimensional embedding, the restriction map \(p:\mathfrak h^*\to \mathfrak h_0^*\) satisfies
\[
p(\alpha_1)=p(\alpha_3)=p(\alpha_4)=p(\alpha_6)=\beta_1,\qquad
p(\alpha_2)=p(\alpha_5)=\beta_2.
\tag{8}
\]
This follows immediately from the weights of the \(7\)-dimensional representation of \(G_2\).

We set \(X_i=e^{-\alpha_i}\), \(x=e^{-\beta_1}\), \(y=e^{-\beta_2}\). By \cite[Section~4.1, Lemma~4.4]{unique_branch}, the hypothesis
\(\operatorname{res}_{G_2}\pi_\lambda \simeq \operatorname{res}_{G_2}\pi_\mu\) is equivalent to
\(u_\lambda=u_\mu\), where $u_\lambda, u_\mu$ are the projections of the respective normalized Weyl numerators.

\begin{enumerate}
\item Setting \(y=0\) in \(u_\lambda=u_\mu\) isolates the subsystem generated by
\(\{\alpha_1,\alpha_3,\alpha_4,\alpha_6\}\), which is of type
\(A_1\times A_2\times A_1\). This yields
\[
\{n_1,n_6,n_3,n_4,n_3+n_4\}
=
\{m_1,m_6,m_3,m_4,m_3+m_4\}.
\tag{9}
\]
\item Setting \(x=0\) in \(u_\lambda=u_\mu\) isolates the subsystem generated by
\(\{\alpha_2,\alpha_5\}\), of type \(A_1\times A_1\), giving
\[
\{n_2,n_5\}=\{m_2,m_5\}.
\tag{10}
\]
\end{enumerate}

Since \(G_2\) contains the principal \(\SL_2(\mathbb C)\) of \(\SL_7(\mathbb C)\), equation
   \eqref{principal_numbers}  implies the equality of multisets of all consecutive sums:
\[
\{\,n_i+n_{i+1}+\cdots+n_j : 1\le i\le j\le 6\,\}
=
\{\,m_i+m_{i+1}+\cdots+m_j : 1\le i\le j\le 6\,\}.
\tag{11}
\]
We also have \[\sum_{i=1}^6 n_i=\sum_{i=1}^6 m_i. \tag{12}\]
 Using (9), (10), (11) and (12) and following the same arguments as in the previous cases
we get that either $\pi_\lambda\simeq \pi_\mu$
or $\pi_\lambda\simeq \pi_\mu^\vee$. 
 \end{proof}

The following result addresses the classical embedding $\mathrm{SO}_{2n} \subset \mathrm{SL}_{2n}$ and describes the relationship between irreducible representations under restriction.

\begin{corollary}
Let $\pi_1$ and $\pi_2$ be irreducible finite-dimensional representations of
 $\mathrm{SL}_{2n}(\mathbb C)$ such that
\[
\operatorname{res}_{\mathrm{SO}_{2n}}\pi_1\simeq
\operatorname{res}_{\mathrm{SO}_{2n}}\pi_2.
\]
Then either $\pi_1\simeq\pi_2$ or $\pi_1\simeq\pi_2^\vee$ (the contragredient).
\end{corollary}

 \begin{proof}
     Note that there exists a Cartan subalgebra $\mathfrak{h}_0$
     contained in both $\mathfrak{so}_{2n}$ and 
     $\mathfrak{sp}_{2n}$. Let $\mathfrak h\subset\mathfrak{sl}_{2n}$
be any Cartan subalgebra containing $\mathfrak h_0$, and set
$T_0:=\exp(\mathfrak h_0)=\T\cap\SO_{2n}=\T\cap\Sp_{2n}$, where
$\T=\exp(\mathfrak h)$. The hypothesis $\operatorname{res}_{\SO_{2n}}\pi_1\simeq\operatorname{res}_{\SO_{2n}}\pi_2$ implies the characters of $\pi_1$ and $\pi_2$ agree on $T_0$, hence their restrictions to $\Sp_{2n}$ also have equal character on the maximal torus $T_0$ of $\Sp_{2n}$.  Therefore the two $\SL_{2n}$--representations restrict isomorphically to $\Sp_{2n}$. Applying the result established for the pair $(\SL_{2n},\Sp_{2n})$ therefore yields that the highest weights of $\pi_1$ and $\pi_2$ coincide up to the nontrivial diagram involution of type $A_{2n-1}$, i.e. either $\pi_1\simeq\pi_2$ or $\pi_1\simeq\pi_2^\vee$, as required.
\end{proof}

\begin{remark}
In the cases of the embeddings
\[
G_2\subset \mathrm{Spin}_7(\mathbb C)
\qquad\text{and}\qquad
G_2\subset \mathrm{SO}_8(\mathbb C),
\]
the preceding arguments prove a much stronger statement. Namely, for
\(\mathfrak{spin}_7\), restriction to a principal
\(\mathfrak{sl}_2\) determines an irreducible representation uniquely,
while for \(\mathfrak{so}_8\), restriction to a principal
\(\mathfrak{sl}_2\) determines an irreducible representation up to an
outer automorphism. Thus, in both cases, the subgroup \(G_2\) is not
needed for the final separation of representations: the principal
\(\mathfrak{sl}_2\) already detects the highest weight, uniquely in type
\(B_3\) and up to triality in type \(D_4\).

However, in case (8) of the embedding \(G_2\subset \mathrm{SL}_7(\mathbb C)\), the situation is different: the restriction to the principal \(\SL_2(\mathbb C)\) does \emph{not} generally determine an irreducible representation of \(\SL_7(\mathbb C)\) up to outer automorphisms. Using Theorem \ref{prod}, one can see that for any simple group \(G\) of rank at most \(4\), an irreducible representation of \(G\) is uniquely determined by its restriction to the principal \(\SL_2(\mathbb C)\) up to an outer automorphism.
\end{remark}

A direct consequence of Theorem~\ref{res_thm} is the following:

\begin{corollary}
For the pairs \((\G,\G_0)\) considered in Theorem~\ref{res_thm}, the
restriction map on isomorphism classes
\[
\operatorname{Irr}(\G)\longrightarrow \operatorname{IrrRep}(\G_0),
\qquad
[V]\longmapsto [\operatorname{res}_{\G_0}^{\G}V],
\]
is injective modulo the subgroup
of automorphisms of \(\G\) that act trivially on \(\G_0\). 
\end{corollary}

\section{Branching for diagonal embeddings}\label{two-copies}
Let \(H\subset G\) be one of Dynkin's pairs with $H \subset G$ containing a regular unipotent
element of $G$. In this section we consider
the diagonal embedding
\[
H\hookrightarrow G\times G,
\] 
and ask the same uniqueness question: given two pairs of dominant weights
$(\lambda_1,\lambda_2)$ and $(\mu_1,\mu_2)$ of $G$, when does
\[
\operatorname{res}_{H}\!\big(V(\lambda_1)\otimes V(\lambda_2)\big)
\;\simeq\;
\operatorname{res}_{H}\!\big(V(\mu_1)\otimes V(\mu_2)\big)
\]
force $(\lambda_1,\lambda_2)$ and $(\mu_1,\mu_2)$ to coincide up to the
natural symmetries (swapping the two factors and applying an outer
automorphism of $G$ that fixes $H$) ?

For the diagonal pairs $(\mathrm{SO}_{2k}(\mathbb C) \times \mathrm{SO}_{2k}(\mathbb C),\mathrm{SO}_{2k-1}(\mathbb C))$, 
$(E_6 \times E_6,F_4)$ and $(Spin_8(\mathbb C) \times Spin_8(\mathbb C), G_2)$ we prove that the answer is affirmative: the branching
restriction uniquely determines the pair of representations up to the
expected symmetries (Theorems~\ref{so2k_diagonal}, ~\ref{e6-diagonal}, and ~\ref{spin-diagonal}). The proofs rely on an induction on the rank that reduces the problem to a small set of remaining parameters, followed by a comparison of the multisets of exponents coming from the restriction to the principal $\mathrm{SL}_2(\mathbb C)$ (which factors through the diagonal $H$). 

For the remaining diagonal pairs, principal
$\mathrm{SL}_2(\mathbb C)$-data alone is not sufficient: we give explicit counterexamples showing that
the multiset equality of principal $\mathrm{SL}_2(\mathbb C)$ exponents can hold for
two pairs of weights that are neither equal nor swapped, nor related by
the contragredient automorphism.  Thus, for those pairs, a finer analysis
involving the full group $G_0$ is required; for example, see \cite{unique_branch} for a proof for the pair $(\mathrm{SL}_{2k}(\mathbb C) \times \mathrm{SL}_{2k}(\mathbb C),\mathrm{Sp}_{2k}(\mathbb C))$.

\begin{theorem}\label{so2k_diagonal}
Let $G = \mathrm{SO}_{2k}(\mathbb C) \times \mathrm{SO}_{2k}(\mathbb C)$ and let $H \subset G$ be the diagonal embedding of $\mathrm{SO}_{2k-1}(\mathbb C)$. 
Suppose $(\lambda_1, \lambda_2)$ and $(\mu_1, \mu_2)$ are pairs of dominant weights of $\mathrm{SO}_{2k}(\mathbb C)$ such that
\[
\mathrm{res}_{H} \big( V(\lambda_1) \otimes V(\lambda_2) \big) \simeq 
\mathrm{res}_{H} \big( V(\mu_1) \otimes V(\mu_2) \big).
\]
Then, after possibly swapping factors and applying the non‑trivial diagram automorphism of $\mathrm{SO}_{2k}(\mathbb C)$ to one or both factors, we have 
\[
\lambda_1 = \mu_1 \quad \text{and} \quad \lambda_2 = \mu_2.
\]
\end{theorem}

\begin{proof}
Write
\[
n_i=\langle\lambda_1+\rho,\alpha_i^\vee\rangle,\qquad
n_i'=\langle\lambda_2+\rho,\alpha_i^\vee\rangle,
\]
and similarly
\[
m_i=\langle\mu_1+\rho,\alpha_i^\vee\rangle,\qquad
m_i'=\langle\mu_2+\rho,\alpha_i^\vee\rangle.
\]
Now we decompose the simple roots of $\mathfrak{g}\times\mathfrak{g}$ into 
$\Delta_1$ (all simple roots except the two copies of $\alpha_1$) 
and $\Delta_2$ (the two copies of $\alpha_1$). 

Arguing exactly as in the proof of Theorem~\ref{res_thm}, more precisely
as in the step leading to equation \ref{tail}, we apply the same restriction
argument to the subsystem
$\Delta_1=\Delta\setminus\{\alpha_1\}$
in each of the two factors. The corresponding smaller diagonal pair is
again of the same form, so the induction hypothesis applies. Therefore,
after possibly interchanging the two tensor factors and applying the
diagram automorphism of \(D_k\) to one or both factors, we may assume
\[
n_i=m_i,\qquad n_i'=m_i'
\qquad (2\le i\le k).
\]

Since the principal $\mathrm{SL}_2(\mathbb C)$ factors through $H$, the
hypothesis $\operatorname{res}_{H}(V(\lambda_1)\otimes V(\lambda_2))
\simeq \operatorname{res}_{H}(V(\mu_1)\otimes V(\mu_2))$ implies an
equality of the multisets of exponents for the tensor product, exactly as
in equation ~\ref{principal_numbers}.  Using the equalities $n_i=m_i$ and $n_i'=m_i'$
for all $i\ge 2$ obtained from the induction step, these two multisets
differ only in the entries that involve $n_1,n_1'$ versus $m_1,m_1'$.
Comparing the two multisets therefore forces
$\{n_1,n_1'\} = \{m_1,m_1'\}$.

If \(n_1=m_1\), we are done. If \(n_1=m_1'\), then \(n_1'=m_1\). Comparing
the smallest terms involving the adjacent root \(\alpha_2\) gives
\[
\min\{n_1+n_2,n_1'+n_2'\}
=
\min\{n_1'+n_2,n_1+n_2'\}.
\]
Since \(n_1\neq n_1'\), this forces \(n_2=n_2'\). Repeating the same
argument along the chain gives \(n_i=n_i'\) for all \(i\ge2\). Hence the
two weights are equal after swapping the tensor factors.
\end{proof}

\begin{theorem}\label{e6-diagonal}
Let $G = E_6 \times E_6$ and let $H \subset G$ be the diagonally embedded $F_4$. 
Suppose $(\lambda_1, \lambda_2)$ and $(\mu_1, \mu_2)$ are pairs of dominant weights of $E_6$ such that
\[
\mathrm{res}_{H} (V(\lambda_1) \otimes V(\lambda_2)) \simeq \mathrm{res}_{H} (V(\mu_1) \otimes V(\mu_2)).
\]
Then, up to possibly swapping the two factors and applying the non-trivial outer automorphism of $E_6$ fixing $F_4$, we have
\[
(\lambda_1, \lambda_2) = (\mu_1, \mu_2) \quad \text{or} \quad (\lambda_1, \lambda_2) = (\mu_2, \mu_1).
\]
\end{theorem}

\begin{proof}
We regard $V(\lambda_1)\otimes V(\lambda_2)$ as the irreducible representation
of $\G$ with highest weight $(\lambda_1,\lambda_2)$, and similarly for
$(\mu_1,\mu_2)$.  The subgroup $H$ contains a
regular unipotent element of $\G$; therefore the hypothesis
\[
\operatorname{res}_{H}\!\big(V(\lambda_1)\otimes V(\lambda_2)\big)
\simeq
\operatorname{res}_{H}\!\big(V(\mu_1)\otimes V(\mu_2)\big)
\]
allows us to apply the restriction argument of
Theorem~\ref{res_thm} (more precisely, equation \ref{tail}) to the pair $(\G,H)$.

We write the simple roots of $\G$ as two copies
$\{\alpha_i^{(1)}\}$ and $\{\alpha_i^{(2)}\}$ ($i=1,\dots,6$) of the
simple roots of $E_6$ and we decompose the set of simple roots as
\[
\Delta_1 = \{\text{all simple roots except } \alpha_2^{(1)},\alpha_2^{(2)}\},
\qquad
\Delta_2 = \{\alpha_2^{(1)},\alpha_2^{(2)}\}.
\]
The subalgebra corresponding to $\Delta_1$ is $\mathfrak{a}_5\times\mathfrak{a}_5$,
the product of two copies of the $A_5$ Levi subalgebra of $E_6$ (obtained by
omitting $\alpha_2$).  Let $\lambda_1^{(1)},\lambda_2^{(1)}$ denote the highest
weights of the restrictions of $V(\lambda_1)$ and $V(\lambda_2)$ to this
$A_5$ subalgebra, expressed in terms of the fundamental weights of $A_5$ and we define $\mu_1^{(1)},\mu_2^{(1)}$ analogously. Write \[n_i=\langle\lambda_1+\rho,\alpha_i^\vee\rangle,\qquad
n_i'=\langle\lambda_2+\rho,\alpha_i^\vee\rangle,
\]
and similarly \(m_i,m_i'\).

Arguing as in the proof of Theorem~\ref{res_thm}, specifically in the
step leading to \eqref{tail}, we apply the same restriction argument to
the decomposition
$
\Delta=\Delta_1\sqcup\Delta_2.$ Since the subsystem corresponding to \(\Delta_1\) is of type
\(A_5\times A_5\), the induction hypothesis gives, after the allowed
symmetries,
\[
n_i=m_i,\qquad n_i'=m_i'
\qquad (i\neq 2).
\]
It remains to compare the two remaining coordinates. Applying the same
argument to the subsystem corresponding to \(\Delta_2\) (the corresponding root subsystem is of type \(A_1\times A_1\)), we get
\[
(1-z^{n_2})(1-z^{n_2'})
=
(1-z^{m_2})(1-z^{m_2'}),
\]
and hence
\[
\{n_2,n_2'\}=\{m_2,m_2'\}.
\]
If the matching is factorwise, then all shifted coordinates agree and we
are done. If the matching is crossed, say \(n_2=m_2'\) and
\(n_2'=m_2\), then either \(n_2=n_2'\), in which case there is nothing to
prove, or \(n_2\neq n_2'\).

In the latter case, we use the equality of the principal
\(\mathrm{SL}_2(\mathbb C)\)-exponent multisets in equation \ref{principal_numbers}. Comparing the exponents
which contain the coordinate \(n_2\), and using the already established
equalities \(n_i=m_i\), \(n_i'=m_i'\) for \(i\neq2\), forces the remaining
coordinates in the two factors to agree. Thus, after interchanging the two
tensor factors, we are reduced to the factorwise matching case.
Therefore
\[
(\lambda_1,\lambda_2)=(\mu_1,\mu_2)
\]
up to swapping the two factors and applying the  outer
automorphism of \(E_6\).
\end{proof}

\begin{theorem}\label{spin-diagonal}
Let $G = Spin_8(\mathbb C) \times Spin_8(\mathbb C)$ and let $H \subset G$ be the diagonally embedded $G_2$. 
Suppose $(\lambda_1, \lambda_2)$ and $(\mu_1, \mu_2)$ are pairs of dominant weights of $Spin_8(\mathbb C)$ such that
\[
\mathrm{res}_{H} V(\lambda_1) \otimes V(\lambda_2) \simeq \mathrm{res}_{H} V(\mu_1) \otimes V(\mu_2).
\] Then \((\lambda_1,\lambda_2)\) and \((\mu_1,\mu_2)\) agree up to swapping
the two factors and applying triality automorphisms.
\end{theorem}

\begin{proof}
Use Bourbaki numbering for \(D_4\), with \(\alpha_2\) the central root and
\(\alpha_1,\alpha_3,\alpha_4\) the outer roots. Write
\[
n_i=\langle \lambda_1+\rho,\alpha_i^\vee\rangle,\qquad
n_i'=\langle \lambda_2+\rho,\alpha_i^\vee\rangle,
\]
and similarly \(m_i,m_i'\) for \(\mu_1,\mu_2\).

Let \(p\) be the restriction map to the \(G_2\)-root lattice. Under
triality folding,
\[
p(\alpha_1)=p(\alpha_3)=p(\alpha_4)=\beta_1,\qquad
p(\alpha_2)=\beta_2.
\]
Set \(x=e^{-\beta_1}\), \(y=e^{-\beta_2}\). The hypothesis gives
\[
p(U_{\lambda_1})p(U_{\lambda_2})
=
p(U_{\mu_1})p(U_{\mu_2}).
\tag{1}
\]

Setting \(x=0\) in (1) gives
\[
(1-y^{n_2})(1-y^{n_2'})
=
(1-y^{m_2})(1-y^{m_2'}),
\]
hence $
\{n_2,n_2'\}=\{m_2,m_2'\}.$
After swapping the two tensor factors if necessary, assume
\[
n_2=m_2,\qquad n_2'=m_2'.
\tag{2}
\]

Next, setting \(y=0\) in (1) gives
\[
\prod_{i\in\{1,3,4\}}(1-x^{n_i})
\prod_{i\in\{1,3,4\}}(1-x^{n_i'})
=
\prod_{i\in\{1,3,4\}}(1-x^{m_i})
\prod_{i\in\{1,3,4\}}(1-x^{m_i'}).
\]
Therefore
\[
\{n_1,n_3,n_4,n_1',n_3',n_4'\}
=
\{m_1,m_3,m_4,m_1',m_3',m_4'\}.
\tag{3}
\]
Since the principal homomorphism 
\(\mathrm{SL}_2(\mathbb{C}) \to \G\) factors through $H$,  
the equality of multisets \eqref{principal_numbers} holds for $(\lambda_1,\lambda_2)$ and $(\mu_1,\mu_2)$ and we get \[
\mathcal P(\lambda_1)\sqcup\mathcal P(\lambda_2)
=
\mathcal P(\mu_1)\sqcup\mathcal P(\mu_2),
\tag{4}
\] where 
\[
\begin{aligned}
\mathcal P(\lambda_1)=
\{&
n_1,n_2,n_3,n_4,\,
n_1+n_2,n_2+n_3,n_2+n_4,\\
&
n_1+n_2+n_3,\,
n_1+n_2+n_4,\,
n_2+n_3+n_4,\\
&
n_1+n_2+n_3+n_4,\,
n_1+2n_2+n_3+n_4
\}.
\end{aligned}
\] and similarly for other weights. Using (2) and (3), all terms in (4) except the triple sums
$n_i+n_j+n_2$ (equivalently, the pairwise sums $n_i+n_j$ of the outer coordinates, up to the known shift by $n_2$) are already determined. So we get \[
\Sigma_2(\{n_1,n_3,n_4\})\sqcup \Sigma_2(\{n_1',n_3',n_4'\})
=\Sigma_2(\{m_1,m_3,m_4\})\sqcup \Sigma_2(\{m_1',m_3',m_4'\}),
\]
where $\Sigma_2(\{r,s,t\})=\{r+s,\ r+t,\ s+t\}.$
A multiset of three numbers is determined by its pairwise sums. Hence
\[
\bigl\{\{n_1,n_3,n_4\},\{n_1',n_3',n_4'\}\bigr\}
=
\bigl\{\{m_1,m_3,m_4\},\{m_1',m_3',m_4'\}\bigr\}.
\]

Together with (3), this shows that the two pairs of shifted weights agree
up to interchanging the two factors and permuting the three outer
\(D_4\)-nodes. These permutations are precisely the triality
automorphisms. Therefore
$
(\lambda_1,\lambda_2)$
and
$
(\mu_1,\mu_2)$
agree up to the claimed symmetries.
\end{proof}

\subsection{Insufficiency of the principal 
\texorpdfstring{${\rm SL}_2(\mathbb{C})$}{} alone}

For diagonal branching problems, the restriction to the principal
${\rm SL}_2(\mathbb{C})$ is a useful necessary test, but it is not
sufficient in general to determine the pair of highest weights up to the
expected symmetries. We illustrate this phenomenon by the following two explicit type
\(A\) examples: 
\[
\Delta SO_5(\mathbb C)\subset
\mathrm{SL}_5(\mathbb C)\times \mathrm{SL}_5(\mathbb C), \,\,\,\, 
\Delta Sp_6(\mathbb C)\subset
\mathrm{SL}_6(\mathbb C)\times \mathrm{SL}_6(\mathbb C).
\]
Similar failures occur for the other diagonal pairs not treated by our
positive results, including the non-folding \(G_2\)-embeddings. 

For a tuple \(\mathbf a=(a_1,\dots,a_k)\) of positive integers, define the
multiset
\[
N(\mathbf a)
=
\{
a_1,\,
a_1+a_2,\,
\dots,\,
a_1+\cdots+a_k,\,
a_1+\cdots+a_k,\,
a_2+\cdots+a_k,\,
\dots,\,
a_k
\}.
\]
Thus \(N(\mathbf a)\) records the principal
\(\mathrm{SL}_2(\mathbb C)\)-exponents occurring in the reduced type
\(A\) situation.

\begin{enumerate}
\item
\textbf{The case of \(\mathrm{SL}_5(\mathbb C)\).}
Let $\mathbf n=(1,1,2,3),
\mathbf m=(2,1,3,1),
\mathbf n'=(3,1,2,1),
\mathbf m'=(1,1,3,2).$
A direct computation gives
\[
N(\mathbf n)\sqcup N(\mathbf m)
=
\{1,1,\;2,2,\;3,3,\;4,4,\;5,5,\;6,6,\;7,7,7,7\}
=
N(\mathbf n')\sqcup N(\mathbf m').
\]
However, the pair \((\mathbf n,\mathbf m)\) cannot be obtained from
\((\mathbf n',\mathbf m')\) by interchanging the two tuples or by
reversing the individual tuples.

\item
\textbf{The case of \(\mathrm{SL}_6(\mathbb C)\).}
Let $
\mathbf n=(1,1,1,2,3),
\mathbf m=(2,1,1,3,1),
\mathbf n'=(1,1,1,3,2),
\mathbf m'=(3,1,1,2,1).$
Then
\[
\begin{aligned}
N(\mathbf n)\sqcup N(\mathbf m)
=
\{&
1,1,\;
2,2,\;
3,3,3,\;
4,4,\;
5,5,5,\;
6,6,\;
7,7,\;
8,8,8,8
\} \\
=
&\;
N(\mathbf n')\sqcup N(\mathbf m').
\end{aligned}
\]
Again, the pair \((\mathbf n,\mathbf m)\) is not related to
\((\mathbf n',\mathbf m')\) by interchanging the two tuples or by
reversing the individual tuples.
\end{enumerate}

These examples show that equality of the principal
\(\mathrm{SL}_2(\mathbb C)\)-exponent multisets does not, by itself,
determine the pair of highest weights up to the natural symmetries.
Thus, in the diagonal setting, the principal \(\mathrm{SL}_2(\mathbb C)\)-restriction
should be regarded only as a necessary invariant. A proof of diagonal
uniqueness, where it holds, must use finer information from the full
restricted character.


\begin{thebibliography}{Bou75}


\bibitem[Bo05]{Bo} N. Bourbaki, {\it Lie groups and Lie algebras. Chapters 7-9.}
Translated from the 1975 and 1982 French originals by Andrew Pressley. Elements of Mathematics (Berlin). Springer-Verlag, Berlin, 2005.
 
  \bibitem[Bo 81]{bourbaki_lie_4_6}
N.~Bourbaki.
\newblock {\em \'{E}l\'{e}ments de math\'{e}matique}.
\newblock Masson, Paris, 1981.
\newblock Groupes et alg\`ebres de Lie. Chapitres 4, 5 et 6. [Lie groups and
  Lie algebras. Chapters 4, 5 and 6].

  
\bibitem[Gro00]{gross_miniscule}
Benedict~H. Gross, \emph{On minuscule representations and the principal {${\rm
  SL}_2$}}, Represent. Theory \textbf{4} (2000), 225--244. \MR{1795753}

  \bibitem[GL16]{guilhot_lecouvey}
J. Guilhot and C. Lecouvey, \emph{Isomorphic induced modules and Dynkin
diagram automorphisms of semisimple Lie algebras}, Glasg. Math. J.
\textbf{58} (2016), no.~1, 187-203.

\bibitem[Kos76]{kostant_eta}
B. Kostant, \emph{On {M}acdonald's {$\eta $}-function formula, the
  {L}aplacian and generalized exponents}, Advances in Math. \textbf{20} (1976),
  no.~2, 179--212. \MR{485661}

  \bibitem[Kos59]{kostant_principal}
B. Kostant, \emph{The principal three-dimensional subgroup and the
Betti numbers of a complex simple Lie group}, Amer. J. Math.
\textbf{81} (1959), no.~4, 973-1032.


\bibitem[NP22]{unique_branch}
Santosh Nadimpalli and Santosha Pattanayak, \emph{On uniqueness of branching to
  fixed point {L}ie subalgebras}, Forum Math. \textbf{34} (2022), no.~6,
  1663--1678. \MR{4504111}

\bibitem[NPP25]{character}
Santosh Nadimpalli, Santosha Pattanayak and Dipendra Prasad, \emph{Character theory at a torsion element},  	arXiv:2504.14684 (2025).

\bibitem[Pra16]{dipendra_half_sum}
Dipendra Prasad, \emph{Half the sum of positive roots, the {C}oxeter element,
  and a theorem of {K}ostant}, Forum Math. \textbf{28} (2016), no.~1, 193--199.
  \MR{3441113}

\bibitem[Raj04]{rajan_1}
C.~S. Rajan, \emph{Unique decomposition of tensor products of irreducible
  representations of simple algebraic groups}, Ann. of Math. (2) \textbf{160}
  (2004), no.~2, 683--704. \MR{2123935}
\end{thebibliography}
\end{document}